\documentclass[a4paper, 12pt]{amsart}
\usepackage[left=2.5cm, right=2.5cm, top=4cm]{geometry}

\usepackage{amsfonts}
\usepackage{amsthm}
\usepackage{graphicx}

\usepackage{tikz}

\usepackage{graphicx, amssymb, latexsym, amsfonts, amsmath, lscape, amscd,
amsthm, color, epsfig}

\usepackage[T1]{fontenc}
\usepackage[latin1]{inputenc}

\usepackage[babel=true]{csquotes}

\usepackage{setspace}

\newtheorem{theorem}{Theorem}

\newtheorem{corollary}[theorem]{Corollary}

\newtheorem{lemma}[theorem]{Lemma}
\newtheorem{proposition}[theorem]{Proposition}
\newtheorem{problem}[theorem]{Problem}
\newtheorem{remark}[theorem]{Remark}

\newtheorem{definition}{Definition}

\usepackage{datetime}
\newdate{date}{16}{5}{2019}
\date{\displaydate{date}}

\author{{Amine El Sahili}$^1$ \and {Zeina Ghazo Hanna}$^1$}
\address{Lebanese University, Beirut, Lebanon}
\email{sahili@ul.edu.lb ; zeina$\_$hanna$\_$93@live.com}

\title[Oriented Hamiltonian Paths and Cycles]{About the Number of Oriented Hamiltonian Paths and Cycles in Tournaments}

\begin{document}

\maketitle
\begin{abstract}
We prove that a tournament $T$ and its complement $\overline{T}$ contain the same number of oriented Hamiltonian paths (resp. cycles) of any given type, as a generalization of Rosenfeld's result proved for antidirected paths.
\end{abstract}

\medskip
\section{Introduction}

An oriented Hamiltonian path in a tournament is an oriented path containing all its vertices, and if this path is directed, then it is said to be a directed Hamiltonian path. The definitions are similar for cycles. Counting Hamiltonian paths and cycles in a tournament is a well known problem. Given a certain type of oriented Hamiltonian paths (resp. cycles), one may ask how many such paths (resp. cycles) can be found in a tournament. No exact value of these numbers was given. What was done in this area is bounding the number of only the directed Hamiltonian paths (resp. cycles) in tournaments, and working on characterizing the tournaments having this minimum or maximum number.
\par The oldest result through this investigation was given by Szele \cite{Szele}, more than seventy years ago, who gave lower and upper bounds for the maximum number $P(n)$ of directed Hamiltonian paths in a tournament on $n$ vertices, $\frac{n!}{2^{n-1}}\leq P(n) \leq c_1\frac{n!}{2^{{\frac{3}{4}}n}},$ where $c_1$ is a positive constant independent of $n$. Then, the upper bound of $P(n)$ was improved by Alon \cite{Alon}: $P(n)\leq c_2.n^{\frac{3}{2}}\frac{n!}{2^{n-1}},$ where $c_2>0$ is independent of $n$. 
For the minimum number of directed Hamiltonian paths in a tournament, we can easily verify that it is equal to $1$, and this value corresponds to the transitive tournament. But in the case of strong tournaments, this number increases a lot, as for the nearly-transitive tournament of order $n$, 
where the number of directed Hamiltonian paths is equal to $2^{n-2}+1$. So in 1972, Moon \cite{Moon} gave upper and lower bounds for the minimum number $h_P(n)$ of directed Hamiltonian paths in strong tournaments of order $n$, and in 2006, after finding a characterization of strong tournaments, Busch \cite{Busch} improved this result by proving that the exact value of this minimum number is exactly equal to the upper bound given by Moon. Later on, Moon and Yang \cite{MoonYang}, constructed some tournaments, called the "special chains", linking between many nearly-transitive tournaments, 
and proved that they contain the minimum number of directed Hamiltonian paths, and that they are the only tournaments verifying this minimum. 
Concerning the maximum and minimum number of directed Hamiltonian cycles (also called Hamiltonian circuits) in tournaments, Thomassen \cite{Thomassen} was able, in 1980, giving an extension of Moon's result previously mentioned, to find the minimum number of these cycles in a 2-connected tournament. 
On the other hand, 
it can be proven, using the probabilistic methods, that the maximum number of directed Hamiltonian cycles in a tournament of order $n$ is greater than $\frac{(n-1)!}{2^n}$. However, Moon observed that it seems difficult to give explicit tournaments with such a large number of directed Hamiltonian cycles.
\par An antidirected path is an oriented path whose arcs have successively opposite directions. Rosenfeld \cite{Rosenfield2} proved in 1974 that the number of antidirected Hamiltonian paths starting with a forward arc is equal to the number of antidirected Hamiltonian paths starting with a backward arc, in any tournament, which can be stated as: the number of antidirected Hamiltonian paths in any tournament $T$ is equal to the number of antidirected Hamiltonian paths in the complement of $T$, denoted by $\overline{T}$.
\par In this paper, we generalize Rosenfeld's result for any type of oriented Hamiltonian paths, and also for cycles: We prove that a tournament $T$ and its complement $\overline{T}$ contain the same number of oriented Hamiltonian paths (resp. cycles) of any given type. Then we establish this fact for any digraph $H$ whose maximal degree is less than or equal to 2.

\section{Basic definitions and preliminary results}

We will follow in this paper the same definitions given in \cite{ElSahili-AA2}.\\
\linebreak
Let $\alpha=(\alpha_1,\alpha_2,\dots,\alpha_s); \ s\geq 1, \ \alpha_i\in\mathbb{Z}, \ \alpha_i \cdot \alpha_{i+1}<0 \ \forall \ i=1,\dots,s-1$.
\par An oriented path $P$ is said to be of type $P(\alpha_1,\alpha_2,\dots,\alpha_s)$ if $P$ is formed by $s$ blocks (i.e. maximal directed subpaths) $I_1, I_2, \dots, I_s$ such that $length(I_i) = \mid\! I_i \!\mid = \mid\! \alpha_i \!\mid$ and with $x_i, y_i$ being the ends of the block $I_i$, $I_i \cap I_{i+1} = \{y_i\} =
\{x_{i+1}\}$, 
the following condition is verified: $\forall \ i=1,\dots,s, \ \alpha_i>0 \iff I_i \text{ is directed from } x_i \text{ to } y_i$. We note $end(I_i)=\{x_i,y_i\}$, and we write $P = I_1 I_2 \dots I_s$. For $u,v\in I_i$, $I_i[u,v]$ denotes the subpath of $I_i$ of ends $u$ and $v$. 
This notation can be extended by allowing $\alpha_i$ to be $0$, by considering $P(\alpha_1, ..., \alpha_i, 0, \alpha_{i+2}, ..., \alpha_s) = P(\alpha_1,...,\alpha_i+\alpha_{i+2},...,\alpha_s)$ (remark that in this case, $\alpha_i$ and $\alpha_{i+2}$ have the same sign),
$P(0,\alpha_2,...,\alpha_s)=P(\alpha_2,...,\alpha_s)$,
and $P(\alpha_1,...,\alpha_{s-1},0)=P(\alpha_1,...,\alpha_{s-1})$, and we say that $P(\alpha)$ is a standard type of a path $P$ if $\alpha$ contains no zero components. In this paper, we will always consider standard types of paths unless a non-standard type appears in calculations.
\par Note that a path $P=v_1v_2\dots v_r$ that is of type $P(\alpha_1,\dots,\alpha_s)$ with respect to this enumeration is also of type $P(-\alpha_s,\dots,-\alpha_1)$ with respect to the other enumeration $v_rv_{r-1}\dots v_1$ denoting it, so we remark that any path has at most two types. Moreover, two paths $P$ and $P^{'}$ in a tournament $T$ are said to be equal if they have the same set of arcs, i.e. $E(P) = E(P^{'})$.
\par For $\alpha = (\alpha_1,\dots,\alpha_s)$ in $\mathbb{Z}^s$, we denote by $- \alpha$ the tuple $(-\alpha_1,\dots,-\alpha_s)$ and by
$\overline{\alpha}$ the tuple $(\alpha_s,\alpha_{s-1}\dots,\alpha_1)$. Let $T$ be a tournament, then $\mathcal{P}_T(\alpha_1,\dots,\alpha_s)$ is defined to be the set of oriented paths in $T$ of type $P(\alpha_1,\dots,\alpha_s)$ and $f_T(\alpha_1,\alpha_2,\dots,\alpha_s)$ denotes the cardinality of this set. It can be easily verified that:
$$\mathcal{P}_T(\alpha) = \mathcal{P}_T(\beta) \iff \alpha = \beta \text{ or } \alpha = -\overline{\beta}.$$
Let $\alpha=(\alpha_1,\dots,\alpha_s); \ \alpha_i\in\mathbb{Z}, \ \alpha_{i}\cdot\alpha_{i+1}<0 \ \forall \ i=1,\dots,s-1, \ \alpha_s\cdot\alpha_1<0$.
\par An oriented cycle $C$ is said to be of type $C(\alpha_1,\dots,\alpha_s)$ if $C$ is formed by $s$ blocks $I_1, I_2,\dots, I_s$, with  $end(I_i) = \{x_i,y_i\}$, $\mid\! I_i \!\mid = \mid\!\alpha_i\!\mid$ and $I_i\cap I_{i+1} = \{y_i\} = \{x_{i+1}\}, \ 1\leq i\leq s-1$ and $I_s\cap I_1 = \{y_s\} = \{x_1\}$, 
such that  $\forall \ i=1,\dots,s, \ \alpha_i>0 \iff I_i \text{ is directed from } x_i \text{ to } y_i$. We write $C=I_1I_2...I_s$. Note that for cycles, if $s\neq 1$ ($s=1$ is the case of Hamiltonian circuits), then $s$ must be even. As for paths, we may also allow $\alpha_i$ to be $0$ for cycles, by considering $C(\alpha_1,\dots,\alpha_{i-1},0,\alpha_{i+1},\dots,\alpha_s)=C(\alpha_1,\dots,\alpha_{i-1}+\alpha_{i+1},\dots,\alpha_s)$, $C(0,\alpha_2,\dots,\alpha_s)=C(\alpha_2+\alpha_s,\alpha_3,\dots,\alpha_{s-1})$ and 
$C(\alpha_1,\dots,\alpha_{s-1},0)=C(\alpha_1+\alpha_{s-1},\alpha_2,\dots,\alpha_{s-2})$. We say that $C(\alpha)$ is a standard type of a cycle $C$ if $\alpha$ contains no zero components. In this paper, we will also always consider standard types of cycles unless a non-standard type appears in calculations.
\par If $T$ be a tournament, then $\mathcal{C}_T(\alpha_1,\dots,\alpha_s)$ is defined to be the set of oriented cycles of $T$ of type $C(\alpha_1,\dots,\alpha_s)$ and
$g_T(\alpha_1,\dots,\alpha_s)$ denotes the cardinality of this set. We may also verify that:
\begin{eqnarray*}
  \mathcal{C}_T(\alpha) = \mathcal{C}_T(\beta) \iff & & \! \beta = (\alpha_i,\alpha_{i+1},\dots,\alpha_s,\alpha_1,\dots,\alpha_{i-1}) \\
   & & \! \text{or } \beta = (-\alpha_i,-\alpha_{i-1},\dots,-\alpha_1,-\alpha_s,\dots,-\alpha_{i+1}),
\end{eqnarray*}
for some $1\leq i\leq s$.\\
\linebreak
A tuple $\alpha$ is said to be symmetric if $\alpha = - \overline{\alpha}$.
\par An oriented path $P$ (resp. cycle $C$) is said to be symmetric if there exists  a tuple $\alpha$ that is symmetric, such that $P$ (resp. $C$) is of type $P(\alpha)$ (resp. $C(\alpha)$).
Otherwise, the path $P$ (resp. cycle $C$) is not symmetric.\\
\linebreak
Let $T$ be a tournament on $n$ vertices. An oriented 
cycle $C$ in $T$ is said to be generated by an oriented 
path $P = x_1 x_2 \dots x_n$ if $C = P \cup \langle\{x_1,x_n\}\rangle$. We write $C=C_P$.
For more simplicity, we write $uv$ instead of $<\{u,v\}>$.
\par The relation $\mathcal{R}$ defined on the set of oriented 
 paths in $T$ by:
\begin{equation*}
  P \mathcal{R} P^{'} \iff C_P = C_{P^{'}}
\end{equation*}
is an equivalence relation, and so is $\mathcal{R}_{\alpha}$, the restriction of $\mathcal{R}$ on the set $\mathcal{P}_T(\alpha)$.
\par Let $P=v_1v_2\dots v_n$ and $P'$ be two oriented 
paths in a tournament $T$ of order $n$, it can be easily remarked that $P \mathcal{R} P^{'}$ if and only if $P=P'$ or $P' = v_{i}v_{i+1}\dots v_n v_1 v_2\dots v_{i-1}$ for some $2\leq i\leq n$.
\begin{remark}\label{generatedcycles}
Let $P=v_1v_2\dots v_n$ be an oriented 
path in a tournament $T$, of some type $P(\alpha)=P(\alpha_1,\dots,\alpha_s)$, and let $C=C_P$ be the cycle generated by $P$ in $T$ which is of some type $C(\beta)$. We will see what different values $\beta$ could take:
\begin{itemize}
\item \textbf{Case 1:} $s$ is even. Then if $\alpha_1 >0$ (which means $\alpha_s<0$), we have $\beta=(\alpha_1+1,\alpha_2,\dots,\alpha_s)$ or $\beta=(\alpha_1,\dots,\alpha_{s-1},\alpha_s-1)$ whether $(v_n,v_{1})$ or $(v_{1},v_n)$ $\in E(T)$ respectively, while if $\alpha_1 <0$ (i.e. $\alpha_s>0$), then $\beta=(\alpha_1-1,\alpha_2,\dots,\alpha_s)$ or $\beta=(\alpha_1,\dots,\alpha_{s-1},\alpha_s+1)$ whether $(v_1,v_{n})$ or $(v_{n},v_1)$ $\in E(T)$ respectively.
\item \textbf{Case 2:} $s$ is odd. Then if $\alpha_1 >0$ (which means $\alpha_s>0$ also), we have $\beta=(-1,\alpha_1,\dots,\alpha_s)$ or $\beta=(\alpha_s+1+\alpha_1,\alpha_2,\dots,\alpha_{s-1})$ whether $(v_1,v_{n})$ or $(v_{n},v_1)$ $\in E(T)$ respectively, while if $\alpha_1 <0$ (and so is $\alpha_s$), then $\beta=(1,\alpha_1,\dots,\alpha_s)$ or $\beta=(\alpha_s-1+\alpha_1,\alpha_2,\dots,\alpha_{s-1})$ whether $(v_n,v_{1})$ or $(v_{1},v_n)$ $\in E(T)$ respectively.
\end{itemize} 
So we remark that every oriented 
path $P$ in a tournament $T$ may generate 2 types of cycles, that we will denote by $C(\beta)$ and $C(\beta')$ in the latter sections.
\end{remark}
\begin{remark}\label{zzzzzzzz}
If a path $P$ has the type $P(\alpha)=P(\alpha_1,\dots,\alpha_s)$ where $\alpha$ is symmetric, then the cycle $C_P$ generated by $P$ cannot be symmetric.\\
In fact, if $P$ has the type $P(\alpha)=P(\alpha_1,\dots,\alpha_s)$ and $\alpha$ is symmetric, thus $\alpha_1=-\alpha_s$, so $\alpha_1$ and $\alpha_s$ have opposite signs, which means that $s$ should be even. Thus by the previous remark, $C_P$ has one of these types: $C(\alpha_1+1,\dots,\alpha_s)$ or $C(\alpha_1-1,\dots,\alpha_s)$ or $C(\alpha_1,\dots,\alpha_s-1)$ or $C(\alpha_1,\dots,\alpha_s+1)$. But in all these cases, and due to the fact that $\alpha$ is symmetric, $C_P$ cannot be written as a succession of blocks having the type $C(\beta)$ where $\beta$ is symmetric, thus the cycle $C_P$ cannot be symmetric.
\end{remark}

\noindent Let $\alpha = (\alpha_1, \dots, \alpha_s)\in\mathbb{Z}^s$.
\par An integer $1\leq r\leq s$ is said to be a period of $\alpha$ if $[i\equiv j \ (mod \ r) \Rightarrow \alpha_{i_s}=\alpha_{j_s}]$
where $i_s$ is the unique integer in $\{1,2,\dots,s\}$ such that $i\equiv i_s \ (mod \ s)$.
\par Let $r(\alpha) = min\{r; \ r \text{ is a period of } \alpha\}$.
\par It can be shown that $r$ is a period of $\alpha$ $\iff r(\alpha)$ divides $r$, and consequently $r(\alpha)$ divides $s$, since $s$ is a trivial period of $\alpha$.
\par Let $t(\alpha)=\frac{s}{r(\alpha)}$.
\begin{proposition}\label{t}
Let $\alpha = (\alpha_1,\dots,\alpha_s)\in\mathbb{Z}^s$ and $\alpha^{'}=\overline{\alpha}=(\alpha_s,\dots,\alpha_1)$, then $r(\alpha)=r(\overline{\alpha})$ and $t(\alpha)=t(\overline{\alpha})$.
\end{proposition}

\begin{proof}
  Let $\alpha^{'} = (\alpha^{'}_1,\alpha^{'}_2,\dots,\alpha^{'}_s)=(\alpha_s,\dots,\alpha_1)=\overline{\alpha}$,
  and $r(\alpha^{'})=r^{'}, \ r(\alpha)=r$. \\
  Let $l,p$ be two integers such that $p\equiv l \ (mod \ r)$.
  We would like to prove that $\alpha^{'}_{p_s} = \alpha^{'}_{l_s}$.
  We have $\alpha^{'}_{p_s} = \alpha_k; \ k=s-p_s+1$, i.e. $k\equiv -p+1 \ (mod \ s)$ thus $k\equiv -p+1 \ (mod \ r)$.
  Also, $\alpha^{'}_{l_s} = \alpha_j; \ j=s-l_s+1$, i.e. $j\equiv -l+1 \ (mod \ s)$ thus $j\equiv -l+1 \ (mod \ r)$.
  Now $p\equiv l \ (mod \ r) \Rightarrow -p+1 \equiv -l+1 \ (mod \ r) \Rightarrow k\equiv j \ (mod \ r) \Rightarrow
  \alpha_k = \alpha_j \Rightarrow \alpha^{'}_{p_s} = \alpha^{'}_{l_s}$. Thus, $r$ is a period of $\alpha^{'}$ and $r^{'}\leq r$. Similarly, we can prove that $r\leq r^{'}$. Hence, $r=r^{'}$. Consequently, $t(\alpha)=t(\alpha^{'})$.
\end{proof}
We may also remark that $t(\alpha)=t(-\alpha)$, and we can prove that $\forall 1\leq i \leq s$, we have $t(\alpha)=t(\alpha_i,\alpha_{i+1},\dots,\alpha_s,\alpha_1,\dots,\alpha_{i-1})$. As a result, $\forall 1\leq i \leq s$ we have:
\begin{eqnarray*}
t(\alpha)= && t(-\alpha_i,-\alpha_{i+1},\dots,-\alpha_s,-\alpha_1,\dots,-\alpha_{i-1})\\ = && t(\alpha_i,\alpha_{i-1},\dots,\alpha_1,\alpha_s,\dots,\alpha_{i+1}) \\ =&& t(-\alpha_i,-\alpha_{i-1},\dots,-\alpha_1,-\alpha_s,\dots,-\alpha_{i+1}).
\end{eqnarray*}
Let $C=I_1 I_2 \dots I_s$ be an oriented cycle of type $C(\beta)=C(\beta_1,\dots,\beta_s)$ and let $r=r(\beta)$.
For $1\leq i<j\leq s$, $I_i$ and $I_j$ are said to be similar if $j\equiv i \ (mod \ r)$. This is equivalent to say that $j-i$ is a period of $\beta$.
For every $1\leq i\leq s$, there are $t(\beta)-1$ blocks similar to $I_i$. It follows that if $I_i$ and $I_j$ are similar, then $\beta_i=\beta_j$,
and for any non-negative integer $k$, $I_{[i+k]_s}$ and $I_{[j+k]_s}$ are similar.
\par If $C=I_1 I_2 \dots I_s$ is an oriented cycle of type $C(\beta)=C(\beta_1,\dots,\beta_s)$, two vertices $u,v\in C$ are said to be clones (with respect to $C$) if:
\begin{itemize}
  \item $u$ and $v$ belong to similar blocks, say $I_i$ and $I_j$.
  \item $l(I_i[x_i,u]) = l(I_j[x_j,v])$.\\
  It obviously follows that $l(I_i[u,y_i]) = l(I_j[v,y_j])$.
\end{itemize}
If $C\in \mathcal{C}_T(\beta)$, then each vertex of $C$ has $t(\beta)-1$ clones. 

\bigskip

\section{Oriented Hamiltonian paths}

Recall that Rosenfeld \cite{Rosenfield2} proved in 1974 that the number of antidirected Hamiltonian paths starting with a forward arc is equal to the number of antidirected Hamiltonian paths starting with a backward arc, in any tournament. In this section, we will generalize Rosenfeld's result, showing that in a tournament, the number of Hamiltonian paths of a certain type $P(\alpha)$ is equal to the number of Hamiltonian paths of type $P(-\alpha)$:
\begin{theorem}\label{f(a)=f(-a)}
Let $\alpha=(\alpha_1,\dots,\alpha_s)\in \mathbb{Z}^s$; $\alpha_i \cdot \alpha_{i+1} <0$, $\alpha_1 \geq 0$, and let $T$ be a tournament of order $n$; $n=\sum\limits_{i=1}^s \mid \alpha_i \mid +1$. We have:
$$f_T(\alpha)=f_T(-\alpha).$$
\end{theorem}
In order to prove this result, it is more adequate to work on enumerations of oriented paths. To this purpose, we define the following:\\
\linebreak
\noindent Let $\alpha=(\alpha_1,\dots,\alpha_s); \ \alpha_i\in\mathbb{Z}, \ \alpha_{i}\cdot\alpha_{i+1}<0 \ \forall \ i=1,\dots,s-1$, and let $T$ be a tournament on $n\geq\sum\limits_{i=1}^s \mid \alpha_i \mid +1$ vertices.
\begin{definition}
An enumeration $E=v_1v_2\dots v_r$ of some vertices of $T$ is said to be of type $E(\alpha_1,\dots,\alpha_s)$ if the path $P=v_1v_2\dots v_r$ is of type $P(\alpha_1,\dots,\alpha_s)$ with respect to this enumeration. 
\end{definition}

\begin{definition}
Considering the tournament $T$ of order $n$, $\mathcal{E}_T(\alpha_1,\dots,\alpha_s)$ is defined to be the set of enumerations of any $m= \sum\limits_{i=1}^s \mid \alpha_i \mid +1$ vertices of $T$, $m\leq n$, of type $E(\alpha_1,\dots,\alpha_s)$. We denote by $e_T(\alpha_1,\alpha_2,\dots,\alpha_s)$ the cardinality of this set.
\end{definition}

Remark that, unlike the case of paths, where every path has two types, if two enumerations $E$ and $E'$ have different types, then $E\neq E'$. 
In fact, we can easily prove the following property:
\begin{proposition}\label{SetEnumEqual}
Let $T$ be a tournament of order $n$, and $\alpha=(\alpha_1,\alpha_2,\dots,\alpha_s)$, $\beta=(\beta_1,\beta_2,\dots,\beta_{s'})$, $\sum\limits_{i=1}^s |\alpha_i|\leq n$, and  $\sum\limits_{i=1}^{s'} |\beta_i|\leq n$, we have:
$$\mathcal{E}_T(\alpha) = \mathcal{E}_T(\beta) \iff \alpha = \beta.$$
\end{proposition} 

\begin{proposition}\label{PartEnum}
Let $T$ be a tournament of order $n$.\\
The sets  $\mathcal{E}_T(\alpha)$, $\alpha=(\alpha_1,
\alpha_2,\dots,\alpha_s) \in \mathbb{Z}^s$, $\alpha_i.\alpha_{i+1} <0$, $s\geq 1$, with $\sum\limits_{i=1}^s \mid \alpha_i \mid =n-1$, form a partition of the set $\mathcal{E}_T$ of all the enumerations on $n$ vertices of $T$.
\end{proposition}

\begin{proof}
This proposition follows immediately from the fact that the binary relation $R_E$ defined on the set $\mathcal{E}_T$ by $E_1R_EE_2 \Leftrightarrow E_1$ and $E_2$ belong to the same set $\mathcal{E}_T(\alpha)$, is an equivalence relation on $\mathcal{E}_T$ whose equivalences classes are the sets $\mathcal{E}_T(\alpha)$, $\alpha=(\alpha_1,
\alpha_2,\dots,\alpha_s) \in \mathbb{Z}^s$, $\alpha_i.\alpha_{i+1} <0$, $s \geq 1$, and $\sum\limits_{i=1}^s \mid \alpha_i \mid =n-1$.
\end{proof}

Let $\alpha=(\alpha_1,\dots,\alpha_s)\in \mathbb{Z}^s$; $\alpha_i \cdot \alpha_{i+1} <0$, $\alpha_1 \geq 0$, such that $\sum\limits_{i=1}^s \mid \alpha_i \mid =n-1$. 
\begin{proposition}\label{PE}
If $\alpha$ is symmetric, then we have $\mid \mathcal{E}_T(\alpha)\mid =2.\mid \mathcal{P}_T(\alpha)\mid $, while $\mid \mathcal{E}_T(\alpha)\mid =\mid \mathcal{P}_T(\alpha) \mid $ otherwise.
\end{proposition}

The above proposition can be easily verified. In fact, this result follows from the observation that the automorphism group of an oriented path has either order one or two.\\
\linebreak
We may now give the \textbf{Proof of Theorem \ref{f(a)=f(-a)}}:

\begin{proof} First of all remark that if $\alpha$ is symmetric, so is $-\alpha$ and vice versa. Thus to prove that $f_T(\alpha)=f_T(-\alpha)$, and using Proposition \ref{PE}, it is enough to prove that $e_T(\alpha) =e_T(-\alpha)$. \\
\linebreak
The proof will be done by induction on $s$.\\
\linebreak
If $s=1$, $\alpha=(\alpha_1)=(n-1)$ and $-\alpha=(-\alpha_1)=(1-n)$. Since every directed Hamiltonian path $P=v_1v_2\dots v_n$ in $T$ corresponds to two enumerations $E=v_1v_2\dots v_n$ and $E'=v_nv_{n-1}\dots v_1$ of types $E(\alpha)=E(n-1)$ and $E(-\alpha)=E(1-n)$ respectively, and vice versa, thus $\mid \mathcal{E}_T(\alpha)\mid =\mid \mathcal{E}_T(-\alpha)\mid = \mid \mathcal{P}_T(\alpha)\mid$ and we have $e_T(\alpha)=e_T(-\alpha)$.\\
Suppose that the result is true when $\alpha$ has $s$ components, i.e. if $\alpha=(\alpha_1,\dots,\alpha_s)\in \mathbb{Z}^s$; $\alpha_i \cdot \alpha_{i+1} <0$, $\alpha_1\geq 0$, and $T$ is a tournament of order $n=\sum\limits_{i=1}^s \mid \alpha_i \mid +1$, we have: $e_T(\alpha_1,\dots,\alpha_s)=e_T(-\alpha_1,\dots,-\alpha_s)$, and let's prove the result for $s+1$ components.\\ Let $\alpha=(\alpha_1,\dots,\alpha_s,\alpha_{s+1})\in \mathbb{Z}^{s+1}$; $\alpha_i \cdot \alpha_{i+1} <0$, $\alpha_1\geq 0$, and $T$ be a tournament of order $n=\sum\limits_{i=1}^{s+1} \mid \alpha_i \mid +1$.\\
\linebreak
We argue now by induction on $\alpha_1$. If $\alpha_1=0$, then by induction on $s$, we have $$e_T(0,\alpha_2,\dots,\alpha_{s+1})=e_T(\alpha_2,\dots,\alpha_{s+1})=e_T(-\alpha_2,\dots,-\alpha_{s+1})=e_T(0,-\alpha_2,\dots,-\alpha_{s+1}).$$
So suppose that $\alpha_1>0$, and that the result is true when the first component is equal to $\alpha_1 -1$, and let's prove it when the first component is equal to $\alpha_1$.\\
\linebreak
Let $X\subseteq V(T)$ such that $\mid X \mid= \alpha_1 $. Set $T'=T-X$, and let's define the following sets:\\
$A_X=\mathcal{E}_{\langle X\rangle}(\alpha_1 -1) \times\mathcal{E}_{T'}(\alpha_2,\dots,\alpha_{s+1})$,\\
$A'_X=\lbrace (E,E')\in A_X;$ $E=v_1\dots v_{\alpha_1}$, $E'=v_{\alpha_1+1}\dots v_n$, and $(v_{\alpha_1},v_{\alpha_1+1})\in E(T)\rbrace$,\\
$A''_X=\lbrace (E,E')\in A_X;$ $E=v_1\dots v_{\alpha_1}$, $E'=v_{\alpha_1+1}\dots v_n$, and $(v_{\alpha_1+1},v_{\alpha_1})\in E(T)\rbrace$.\\
Obviously we have: $A'_X\cap A''_X=\emptyset$, and $A_X=A'_X\cup A''_X$, thus $\mid A_X\mid=\mid A'_X\mid + \mid A''_X\mid$.\\
\linebreak
Consider the two sets $$\mathcal{E}_X(\alpha_1,\dots,\alpha_{s+1})=\lbrace E=v_1\dots v_{\alpha_1} v_{\alpha_1+1}\dots v_n \in \mathcal{E}_T(\alpha_1,\dots,\alpha_{s+1}); \ \lbrace v_1,\dots,v_{\alpha_1} \rbrace = X \rbrace,$$
\begin{eqnarray*}
\mathcal{E}_X(\alpha_1-1,\alpha_2-1,\alpha_3\dots,\alpha_{s+1})= && \lbrace E=v_1\dots v_{\alpha_1} v_{\alpha_1+1} \dots v_n \in \mathcal{E}_T(\alpha_1-1,\alpha_2-1,\alpha_3,\dots,\alpha_{s+1}); \\
&& \lbrace v_1,\dots,v_{\alpha_1} \rbrace = X \rbrace.
\end{eqnarray*}
We have $$\mid A'_X \mid=\mid  \mathcal{E}_X(\alpha_1,\dots,\alpha_{s+1})\mid \ \text{and} \ \mid A''_X\mid=\mid \mathcal{E}_X(\alpha_1-1,\alpha_2-1,\alpha_3,\dots,\alpha_{s+1})\mid.$$
In fact, consider the two correspondences $f:$ $\mathcal{E}_X(\alpha_1,\dots,\alpha_{s+1})$ $\longrightarrow$ $A'_X$ such that $\forall$ $E=v_1\dots v_{\alpha_1} v_{\alpha_1+1} \dots v_n \in \mathcal{E}_X(\alpha_1,\dots,\alpha_{s+1})$, $f(E)=(v_1\dots v_{\alpha_1},v_{\alpha_1+1}\dots v_n)$, and $f':$ $\mathcal{E}_X(\alpha_1-1,\alpha_2-1,\alpha_3,\dots,\alpha_{s+1})$ $\longrightarrow$ $A''_X$ such that $\forall$ $E=v_1\dots v_{\alpha_1} v_{\alpha_1+1} \dots v_n \in \mathcal{E}_X(\alpha_1-1,\alpha_2-1,\alpha_3,\dots,\alpha_{s+1})$, $f'(E)=(v_1\dots v_{\alpha_1},v_{\alpha_1+1}\dots v_n)$. We can verify that both are bijective mappings.\\ 
Hence, $$\mid A_X\mid =\mid \mathcal{E}_X(\alpha_1,\dots,\alpha_{s+1})\mid + \mid \mathcal{E}_X(\alpha_1-1,\alpha_2-1,\alpha_3,\dots,\alpha_{s+1})\mid.$$\\
\medskip
Now let's consider $-\alpha=(-\alpha_1,\dots,-\alpha_s)$, and let $X\subseteq V(T)$ such that $\mid X \mid= \alpha_1 $. Set $T'=T-X$, and let's also define the following sets:\\
$B_X=\mathcal{E}_{\langle X\rangle}(-\alpha_1 +1) \times\mathcal{E}_{T'}(-\alpha_2,\dots,-\alpha_{s+1})$,\\
$B'_X=\lbrace (E,E')\in B_X;$ $E=v_1\dots v_{\alpha_1}$, $E'=v_{\alpha_1+1}\dots v_n$, and $(v_{\alpha_1+1},v_{\alpha_1})\in E(T)\rbrace$,\\
$B''_X=\lbrace (E,E')\in B_X;$ $E=v_1\dots v_{\alpha_1}$, $E'=v_{\alpha_1+1}\dots v_n$, and $(v_{\alpha_1},v_{\alpha_1+1})\in E(T)\rbrace$.\\
We also have: $B'_X\cap B''_X=\emptyset$, and $B_X=B'_X\cup B''_X$, thus $\mid B_X\mid=\mid B'_X\mid + \mid B''_X\mid$.\\
\linebreak
Consider the two sets $$\mathcal{E}_X(-\alpha_1,\dots,-\alpha_{s+1})=\lbrace E=v_1\dots v_{\alpha_1} v_{\alpha_1+1} \dots v_n \in \mathcal{E}_T(-\alpha_1,\dots,-\alpha_{s+1}); \ \lbrace v_1,\dots,v_{\alpha_1} \rbrace = X \rbrace,$$ 
$
\mathcal{E}_X(-\alpha_1+1,-\alpha_2+1,-\alpha_3,\dots,-\alpha_{s+1})= \\
\lbrace E=v_1\dots v_{\alpha_1} v_{\alpha_1+1} \dots v_n \in \mathcal{E}_T(-\alpha_1+1,-\alpha_2+1,-\alpha_3,\dots,-\alpha_{s+1}); \  \lbrace v_1,\dots,v_{\alpha_1} \rbrace = X \rbrace.
$ \\
\linebreak
Similarly as before, we can prove that $$\mid B'_X \mid=\mid  \mathcal{E}_X(-\alpha_1,\dots,-\alpha_{s+1})\mid \ \text{and} \ \mid B''_X\mid=\mid \mathcal{E}_X(-\alpha_1+1,-\alpha_2+1,-\alpha_3,\dots,-\alpha_{s+1})\mid,$$
Hence $$\mid B_X\mid =\mid \mathcal{E}_X(-\alpha_1,\dots,-\alpha_{s+1})\mid + \mid \mathcal{E}_X(-\alpha_1+1,-\alpha_2+1,-\alpha_3,\dots,-\alpha_{s+1})\mid.$$
On the other hand, we have that $$\mid A_X \mid = e_X(\alpha_1-1).e_{T'}(\alpha_2,\dots,\alpha_{s+1}),$$ $$\mid B_X \mid =  e_X(-\alpha_1+1).e_{T'}(-\alpha_2,\dots,-\alpha_{s+1}),$$ but since we have here less than $s+1$ blocks, thus by induction, $ e_X(\alpha_1-1)= e_X(-\alpha_1+1)$ and $e_{T'}(\alpha_2,\dots,\alpha_{s+1})=e_{T'}(-\alpha_2,\dots,-\alpha_{s+1})$, thus $$\mid A_X \mid = \mid B_X \mid.$$
Also, $\forall$ $\beta =(\alpha_1,\dots,\alpha_{s+1})$ or $(\alpha_1-1,\alpha_2-1,\alpha_3,\dots,\alpha_{s+1})$ or $(-\alpha_1,\dots,-\alpha_{s+1})$ or $(-\alpha_1+1,-\alpha_2+1,-\alpha_3,\dots,-\alpha_{s+1})$, we have $\mathcal{E}_T(\beta)=\cup_{X\subseteq V(T);\mid X \mid = \alpha_1}\mathcal{E}_X(\beta)$. (The union is disjoint since if $X\neq X'$, the enumerations differ).\\
\linebreak
Since $\mid A_X \mid = \mid B_X \mid$, then $\mid \mathcal{E}_X(\alpha_1,\dots,\alpha_{s+1})\mid + \mid \mathcal{E}_X(\alpha_1-1,\alpha_2-1,\alpha_3,\dots,\alpha_{s+1})\mid$ $=\mid \mathcal{E}_X(-\alpha_1,\dots,-\alpha_{s+1})\mid + \mid \mathcal{E}_X(-\alpha_1+1,-\alpha_2+1,-\alpha_3,\dots,-\alpha_{s+1})\mid$. \\Doing the summation over all the sets $X\subseteq V(T)$, $\mid X \mid = \alpha_1$,  we have:\\
$$ \sum\limits_{X\subseteq V(T);\mid X \mid = \alpha_1} \mid \mathcal{E}_X(\alpha_1,\dots,\alpha_{s+1})\mid +\sum\limits_{X\subseteq V(T);\mid X \mid = \alpha_1}\mid \mathcal{E}_X(\alpha_1-1,\alpha_2-1,\alpha_3,\dots,\alpha_{s+1})\mid \ =$$ $$\sum\limits_{X\subseteq V(T);\mid X \mid = \alpha_1}\mid \mathcal{E}_X(-\alpha_1,\dots,-\alpha_{s+1})\mid+\sum\limits_{X\subseteq V(T);\mid X \mid = \alpha_1}\mid \mathcal{E}_X(-\alpha_1+1,-\alpha_2+1,-\alpha_3,\dots,-\alpha_{s+1})\mid,$$thus,
$$\mid \mathcal{E}_T(\alpha_1,\dots,\alpha_{s+1})\mid + \mid \mathcal{E}_T(\alpha_1-1,\alpha_2-1,\alpha_3,\dots,\alpha_{s+1})\mid$$ $$=\mid \mathcal{E}_T(-\alpha_1,\dots,-\alpha_{s+1})\mid + \mid \mathcal{E}_T(-\alpha_1+1,-\alpha_2+1,-\alpha_3,\dots,-\alpha_{s+1})\mid,$$
which implies that $$e_{T}(\alpha_1,\dots,\alpha_{s+1})+e_{T}(\alpha_1-1,\alpha_2-1,\alpha_3, \dots,\alpha_{s+1})$$ $$=e_{T}(-\alpha_1,\dots,-\alpha_{s+1})+e_{T}(-\alpha_1+1,-\alpha_2+1,-\alpha_3,\dots,-\alpha_{s+1}).$$
But by induction, since $\alpha_1-1<\alpha_1$, we have that $e_{T}(\alpha_1-1,\alpha_2-1,\alpha_3, \dots,\alpha_{s+1})=e_{T}(-\alpha_1+1,-\alpha_2+1,-\alpha_3,\dots,-\alpha_{s+1})$. So we finally get $$e_{T}(\alpha_1,\dots,\alpha_{s+1})=e_{T}(-\alpha_1,\dots,-\alpha_{s+1}),$$
which concludes the proof.
\end{proof}

\section{Oriented cycles and generating paths}

Let $T$ be a tournament. In this section we find a relation between $f_T(\alpha)$, $g_T(\beta)$ and $g_T(\beta')$, where $P(\alpha)$ is some type of oriented Hamiltonian paths in $T$, and $C(\beta)$ and $C(\beta')$ are the two types of cycles that can be generated by a path of type $P(\alpha)$ in the tournament $T$, (see Remark \ref{generatedcycles}), and this result will be of great use in the next section.\\
\linebreak
We first start by proving the following theorem:
\begin{theorem}\label{EquClass}
Let $P\in \mathcal{P}_T(\alpha)$ be an oriented Hamiltonian path in a tournament $T$ and let $C_P$ be the cycle generated by $P$ in $T$, of type $C(\beta)$, such that $C_P$ has at least 2 blocks (i.e. $C_P$ is not a circuit). Then if $C_P$ is non-symmetric, we have $\mid\! \overline{P} \!\mid = t(\beta)$, while if $C_P$ is symmetric, then $\mid\! \overline{P} \!\mid = 2.t(\beta)$, where $\overline{P}$ is the equivalence class of $P$ with respect to $\mathcal{R}_{\alpha}$.
\end{theorem}
In \cite{ElSahili-AA2}, one actually proved that if $C_P$ is non-symmetric, then $\mid\! \overline{P} \!\mid = t(\beta)$. In the following, we will present arguments useful for both the symmetric and the non-symmetric types.
\begin{remark}\label{EquClass1.5}
If $P\in \mathcal{P}_T(\alpha)$ is an oriented Hamiltonian path in a tournament $T$ of order $n$, and $C_P$ the cycle generated by $P$ in $T$, such that $C_P$ is a Hamiltonian circuit,  then $\mid\! \overline{P} \!\mid = n$, where $\overline{P}$ is the equivalence class of $P$ with respect to $\mathcal{R}_{\alpha}$.
\end{remark}
In fact, since $C_P$ is a circuit, then $P$ must be a directed path, also every Hamiltonian circuit is generated by exactly $n$ directed Hamiltonian paths, starting each from a vertex of $C_P$.
\par Note that if $C_P$ is a circuit, say of type $C(\beta)$, ($\beta$ in this case is a \underline{singleton}, that is has one component), then $t(\beta)=1$.\\
\linebreak
In order to prove Theorem \ref{EquClass}, we first give the three following lemmas:
\begin{lemma}\label{betaSym}
Let $T$ be a tournament of order $n$, and let $C=v_1v_2\dots v_n$ be a Hamiltonian cycle in $T$. Then $C$ is symmetric if and only if $\forall$ $1\leq i\leq n$, and for every Hamiltonian path $P=v_iv_{i+1}\dots v_nv_1\dots v_{i-1}$, there exists $1\leq i'\leq n$ such that $P=v_iv_{i+1}\dots v_nv_1\dots v_{i-1}$ and $P'=v_{i'}v_{i'-1}\dots v_1v_nv_{n-1}\dots v_{i'+1}$ have the same type with respect to these enumerations.
\end{lemma}
\begin{proof}
For the necessary condition, since $C$ is symmetric, we can suppose without loss of generality that $C=v_1v_2\dots v_n$ is of type $C(\beta_1,\beta_2,\dots,\beta_s)$ $=I_1I_2\dots I_s$ with respect to this enumeration, where $\beta$ is symmetric. We have $\mid I_j \mid = \mid \beta_j \mid$, and let $end(I_j)=\lbrace x_j,y_j\rbrace$, $\forall$ $1\leq j \leq s$.\\ Suppose that $v_i \in I_j$, for some $1\leq j \leq s$, and suppose without loss of generality that $\beta_j >0$, (the case $\beta_j <0$ is similar).\\
Let $i'=n-(i-2)$, (assuming that if $i=1$, $i'=n+1$ simply denotes $i'=1$), so $P'=v_{n-(i-2)} v_{n-(i-2)-1} \dots v_1v_n \dots v_{n-(i-2)+1}$. We will show that this value of $i'$ satisfies the necessary condition. \\
In fact, let $x=l(I_{j}\left[ x_j,v_i \right])$, then the path $P=v_iv_{i+1}\dots v_nv_1\dots v_{i-1}$ is of type $P(\beta_j-x,\beta_{j+1},\dots, \beta_s,\beta_1,\dots, \beta_{j-1},x-1)$ with respect to this enumeration. \\
Since $\beta$ is symmetric, then $(\beta_1,\beta_2,\dots,
\beta_j) = (-\beta_s,-\beta_{s-1},\dots,-\beta_{s-(j-1)})$, so $|\beta_i|=|\beta_{s-i+1}|$ $\forall$ $1\leq i \leq j$, and since $l(C \left[ v_1v_2\dots v_i\right]) = l(C \left[ v_1v_nv_{n-1}\dots v_{n-(i-2)}\right])$, we deduce that $v_{n-(i-2)}\in I_{s-(j-1)}$ and $l(I_{j}\left[ x_j,v_i \right])=$ $l(I_{s-(j-1)}\left[ y_{s-(j-1)},v_{n-(i-2)} \right])=x$. \\
As a result, the path $P'=v_{n-(i-2)} v_{n-(i-2)-1} \dots v_1v_n \dots v_{n-(i-2)+1}$ is of type $P(-\beta_{s-(j-1)}-x, -\beta_{s-(j-1)-1},\dots,-\beta_1,-\beta_s,\dots,-\beta_{s-(j-1)+1},x-1)$ with respect to this enumeration. \\
But $(\beta_1,\beta_2,\dots,\beta_s)=(-\beta_s,-\beta_{s-1},\dots,-\beta_1)$ (since $\beta$ is symmetric), so we get that $P'=v_{n-(i-2)} v_{n-(i-2)-1} \dots v_1v_n \dots v_{n-(i-2)+1}$ is of type $P(\beta_j-x,\beta_{j+1},\dots, \beta_s,\beta_1,\dots, \beta_{j-1},x-1)$ with respect to this enumeration.\\
\linebreak
For the sufficient condition, suppose to the contrary that $C$ is non-symmetric but $\forall$ $ 1\leq i\leq n$, and for every Hamiltonian path $P=v_iv_{i+1}\dots v_nv_1\dots v_{i-1}$, there exists some $i'$, $1\leq i'\leq n$, such that $P=v_iv_{i+1}\dots v_nv_1\dots v_{i-1}$ and $P'=v_{i'}v_{i'-1}\dots v_1v_nv_{n-1}\dots v_{i'}$ have the same type with respect to these enumerations. \\
Suppose without loss of generality that $v_i\in I_1$, and that $\beta_1>0$. (The case $\beta_1<0$ is similarly treated). The path $P=v_iv_{i+1}\dots v_nv_1\dots v_{i-1}$ is of type $P(\beta_1-x,\beta_2,\dots,\beta_s,x-1)$ with respect to this enumeration, for some $0\leq x \leq \beta_1$. Thus the path $P'=v_{i'}v_{i'-1}\dots v_1v_n\dots v_{i'+1}$ has the type $P(\beta_1-x,\beta_2,\dots,\beta_s,x-1)$ with respect to this enumeration.
But the vertex $v_{i'}$ belongs to $C$, thus $v_{i'}$ belongs to a block $I_j$ of $C$ of length $\mid \beta_j \mid $, then $P'=v_{i'}v_{i'-1}\dots v_1v_n\dots v_{i'+1}$ is of type $P(-\beta_j-y,-\beta_{j-1},\dots,-\beta_1,-\beta_s,\dots,-\beta_{j+1},y-1)$ with respect to this enumeration, where $-\beta_j>0$ in this case (since we should have $\beta_1-x=-\beta_j-y$ and $\beta_1-x>0$), and $0\leq y \leq -\beta_j$. We get $(\beta_1-x,\beta_2,\dots,\beta_s,x-1)=(-\beta_j-y,-\beta_{j-1},\dots,-\beta_1,-\beta_s,\dots,-\beta_{j+1},y-1)$, thus $x-1=y-1$ so $x=y$.\\
As a result, we have: $$(\beta_1-x,\beta_2,\dots,\beta_j,\beta_{j+1},\dots,\beta_s,x-1) \overset{\mathrm{(*)}}{=} (-\beta_j-x,-\beta_{j-1},\dots,-\beta_1,-\beta_s,\dots,-\beta_{j+1},x-1).$$
Property ($*$) implies that $\beta_1-x=-\beta_j-x$ and $\forall$ $2\leq p \leq j$, $\beta_p=-\beta_{j-(p-1)}$, thus $\forall$ $1\leq p \leq j$, $\beta_p=-\beta_{j-(p-1)}$, that is $\beta'=(\beta_1,\beta_2,\dots,\beta_j)=(-\beta_j,-\beta_{j-1},\dots,-\beta_1)$, hence $\beta'$ is symmetric, and we can write $\beta'$ as $(\beta_1,\beta_2,\dots,\beta_{\frac{j}{2}},-\beta_{\frac{j}{2}},\dots,-\beta_2,-\beta_1)$. \\
Also, property ($*$) implies that $\forall$ $1\leq p' \leq s-j$, $\beta_{j+p'}=-\beta_{s-(p'-1)}$, that is $\beta''=(\beta_{j+1},\dots,\beta_{s-1},\beta_s)=(-\beta_s,-\beta_{s-1},\dots,-\beta_{j+1})$, which means that $\beta''$ is symmetric, and we can write $\beta''$ as $(-\beta_s,-\beta_{s-1},\dots,-\beta_{\frac{s-j}{2}},\beta_{\frac{s-j}{2}},\dots,\beta_{s-1},\beta_s)$.\\
So finally we have: 
\begin{eqnarray*}
(\beta_1,\beta_2,\dots,\beta_j,\beta_{j+1},\dots,\beta_s)=&&(\beta_1,\beta_2,\dots,\beta_{\frac{j}{2}},-\beta_{\frac{j}{2}},\dots,-\beta_2,-\beta_1,-\beta_s,-\beta_{s-1}, \\
&& \dots, -\beta_{\frac{s-j}{2}},\beta_{\frac{s-j}{2}},\dots,\beta_{s-1},\beta_s),
\end{eqnarray*}
which is the type of the cycle $C$.\\
Now, if we consider $$\beta^*=(-\beta_{\frac{j}{2}},\dots,-\beta_2,-\beta_1,-\beta_s,-\beta_{s-1},\dots,-\beta_{\frac{s-j}{2}},\beta_{\frac{s-j}{2}},\dots,\beta_{s-1},\beta_s,\beta_1,\beta_2,\dots,\beta_{\frac{j}{2}}),$$
$\beta^*$ is symmetric, and is also a type of the cycle $C$, thus $C$ is symmetric, which leads to a contradiction since $C$ is non-symmetric.
\end{proof}

Now, let $T$ be a tournament of order $n$, and let $C$ be a Hamiltonian cycle in $T$, such that $C=v_1v_2\dots v_n$ is of type $C(\beta)$ with respect to this enumeration, where
$\beta$ is symmetric. We now know by the proof of the necessary condition of Lemma \ref{betaSym} that $\forall$ $1\leq i\leq n$, the Hamiltonian paths $P=v_iv_{i+1}\dots v_nv_1\dots v_{i-1}$ and $P'=v_{n-(i-2)} v_{n-(i-2)-1} \dots v_1v_n \dots v_{n-(i-2)+1}$ have the same type with respect to these enumerations.
\par So let $\mathcal{A}$ be the set of all the paths in $T$ that generate the cycle $C$ and that have the form $v_iv_{i+1}\dots v_nv_1\dots v_{i-1}$ and are of a certain type $P(\alpha)$ with respect to this enumeration, and let $\mathcal{B}$ be the set of all the paths that generate $C$ and that have the form $v_{n-(i-2)} v_{n-(i-2)-1} \dots v_1v_n \dots v_{n-(i-2)+1}$ and that also have the type $P(\alpha)$ with respect to this enumeration.
\begin{lemma}\label{intersection}
We have $\mathcal{A}\cap \mathcal{B}=\emptyset$, and $|\mathcal{A}| =|\mathcal{B}|$.
\end{lemma}
\begin{proof}
Suppose to the contrary that $\exists$ $P \in \mathcal{A}\cap \mathcal{B}$. Since $P\in \mathcal{A}$, then $\exists$ $1\leq i \leq n$ such that $P=v_iv_{i+1}\dots v_nv_1\dots v_{i-2}v_{i-1}$ and is of type $P(\alpha)=P(\alpha_1,\dots,\alpha_s)$ with respect to this enumeration. Since $P\in\mathcal{B}$ also, then $P=v_{i-1}v_{i-2}\dots v_1v_n\dots v_{i+1}v_i$ is of type $P(\alpha)=P(\alpha_1,\dots,\alpha_s)$ with respect to this enumeration, which means that $P=v_iv_{i+1}\dots v_nv_1\dots v_{i-2}v_{i-1}$ is of type $P(-\overline{\alpha})=P(-\alpha_s,\dots,-\alpha_1)$. Thus, $(\alpha_1,\dots,\alpha_s)=(-\alpha_s,\dots,-\alpha_1)$ which means that $\alpha$ is symmetric. But, since $C=C_P$, and since $P$ has the type $P(\alpha)$ where $\alpha$ is symmetric, then the cycle $C$ cannot be symmetric by Remark \ref{zzzzzzzz}, thus $\beta$ cannot be symmetric, which leads to a contradiction since $\beta$ is symmetric.\\
For the second part, consider the correspondence $f:$ $\mathcal{A}$ $\longrightarrow$ $\mathcal{B}$, such that for every $P=v_iv_{i+1}\dots v_nv_1\dots v_{i-1}$ of type $P(\alpha)$ with respect to this enumeration in $\mathcal{A}$, corresponds the path $P'=v_{n-(i-2)} v_{n-(i-2)-1} \dots v_1v_n \dots v_{n-(i-2)+1}$, which belongs to $\mathcal{B}$ since it is of type $P(\alpha)$ with respect to this enumeration, by Lemma \ref{betaSym}. The correspondence $f$ is trivially a bijective mapping, so $|\mathcal{A}|=|\mathcal{B}|$ which concludes the proof.
\end{proof}
The last lemma is a result proven implicitely in \cite{ElSahili-AA2}:
\begin{lemma}\label{clonetype}\cite{ElSahili-AA2}
Let $P=v_1 v_2 \dots v_n$ and $P^{'}=v_i v_{i+1} \dots v_n v_1 \dots v_{i-1}$ be two distinct oriented Hamiltonian paths in a tournament $T$ of order $n$, that generate a cycle $C=C_P=C_{P^{'}}$ in $T$. Then $P$ and $P{'}$ have the same type with respect to these enumerations if and only if $v_1$ and $v_i$ are clones.
\end{lemma}

We may now give the \textbf{Proof of Theorem \ref{EquClass}}:
\begin{proof}
The set $\overline{P}$ contains $P$ as well as the paths $P^{'}$ in $T$ that have the same type as the type of $P$ and such that $C_{P'}=C_P$. Let $C_P=v_1v_2\dots v_n$ such that $C_P$ is of type $C(\beta)$, with respect to this enumeration and $P=v_iv_{i+1}\dots v_nv_1\dots v_{i-1}$ for some $1\leq i \leq n$. If we consider all the paths $P'$ having the form $v_jv_{j+1}\dots v_nv_1\dots v_{j-1}$, and that have the same type as $P$ with respect to these enumerations, and such that $C_P=C_{P'}$, then by Lemma \ref{clonetype},  the number of such paths is exactly the number of clones that an end of $P$ could have, that is $t(\beta)-1$.\\
Let $\mathcal{A}$ be the set of paths that generate $C$ and  have the same type as the type of $P$, following the order $v_1v_2\dots v_n$ of the vertices, and $\mathcal{B}$ be the set of paths that generate $C$ and have the same type as the type of $P$, following the order $v_1v_n\dots v_2$ of the vertices. We have $\mid \mathcal{A} \mid = t(\beta)$. Now let's count the number of paths in $\mathcal{B}$. \\
If the cycle $C$ is non-symmetric, then by Lemma \ref{betaSym}, $\forall 1\leq i\leq n$, the path $P'=v_iv_{i-1}\dots v_1v_n\dots v_{i+1}$ cannot have the same type of $P$ with respect to this enumeration, thus the set $\mathcal{B}$ is empty. As a result, $\mid\overline{P}\mid =\mid\mathcal{A}\mid=t(\beta)$.\\
If the cycle $C$ is symmetric, (we may suppose without loss of generality that the cycle $C=v_1v_2 \dots v_n$ is of type $C(\beta)$ with respect to this enumeration, where $\beta$ is symmetric) then by Lemma \ref{betaSym}, the set $\mathcal{B}$ is non-empty, and by Lemma \ref{intersection} we have that $\mid \mathcal{A} \mid = \mid \mathcal{B}\mid$ and that the sets $\mathcal{A}$ and $\mathcal{B}$ are disjoint, thus we deduce that $\mid\overline{P}\mid = \mid \mathcal{A} \mid + \mid \mathcal{B} \mid= t(\beta)+t(\beta)=2.t(\beta)$. This concludes our proof.
\end{proof}

Note that all of the above results of this section are true for any oriented paths and cycles that are not necessarily Hamiltonian, since any path or cycle defines a set of vertices, and hence a subtournament in which the path and the cycle are Hamiltonian.\\
\linebreak
\noindent The following lemma, proved in \cite{ElSahili-AA2}, is of practical use in the next theorem:

\begin{lemma}\label{equalind}\cite{ElSahili-AA2}
Let $\alpha_1,\alpha_2,\dots,\alpha_s$, $\beta_1,\dots,\beta_s \in \mathbb{Z}$. \\If $(\alpha_1,\alpha_2,\dots,\alpha_s)=(\beta_i,\beta_{i+1},\dots,\beta_s,\beta_1,\dots,\beta_{i-1})$, then for any integer $k\geq 0$, we have
$$
\alpha_{k_s} = \beta_{[k+i-1]_s}.
$$
\end{lemma}

\begin{remark}
We saw in the second case of Remark \ref{generatedcycles} that if $\alpha=(\alpha_1,\dots,\alpha_s)\in \mathbb{Z}^s$, $\alpha_i \cdot \alpha_{i+1}< 0$, $s$ is odd, $T$ is a tournament of order $n=\sum\limits_{i=1}^s \mid \alpha_i \mid +1$, and $P$ be a Hamiltonian path of type $P(\alpha)$ in $T$, then the two types of cycles that can be generated by $P$ have either $s-1$ or $s+1$ blocks. So if we call $C(\beta)$ and $C(\beta')$ these two types, then obviously, the sets $\mathcal{C}_T(\beta)$ and $\mathcal{C}_T(\beta')$ are different. Remark also that when $s$ is odd, $\alpha$ is always non-symmetric, because we can't have $\alpha_1=-\alpha_s$ since $\alpha_1$ and $\alpha_s$ have the same sign.
\end{remark}
However, when $s$ is even, it's a different story. In fact, we have the following result:

\begin{theorem}\label{EqSym}
Let $\alpha=(\alpha_1,\dots,\alpha_s)\in \mathbb{Z}^s$, $\alpha_i \cdot \alpha_{i+1}< 0$, and let $T$ be a tournament of order $n=\sum\limits_{i=1}^s \mid \alpha_i \mid +1$.  We have:
$$\mathcal{C}_T(\beta)=\mathcal{C}_T(\beta')\iff \alpha \  is \ symmetric,$$
\noindent where $C(\beta)$ and $C(\beta')$ are the two types of Hamiltonian cycles in $T$ that can be generated by a Hamiltonian path of type $P(\alpha)$ in $T$.
\end{theorem}

\begin{proof}
The case where $s$ is odd being completely settled by Remark 13, we may assume that $s$ is even.\\
By the first case of Remark \ref{generatedcycles}, when $\alpha_1 >0$, then $\beta=(\beta_1,\beta_2,\dots,\beta_s)=(\alpha_1+1,\alpha_2,\dots,\alpha_s)$ and $\beta'=(\beta'_1,\beta'_2,\dots,\beta'_s)=(\alpha_1,\dots,\alpha_{s-1},\alpha_s-1)$, while if $\alpha_1 <0$, then $\beta=(\beta_1,\beta_2,\dots,\beta_s)=(\alpha_1-1,\alpha_2,\dots,\alpha_s)$ and $\beta'=(\beta'_1,\beta'_2,\dots,\beta'_s)=(\alpha_1,\dots,\alpha_{s-1},\alpha_s+1)$.\\
We will treat the case where $\alpha_1>0$, and the other case is similar.\\
For the sufficient condition, suppose that $\alpha$ is symmetric, thus $\alpha=(\alpha_1,\dots,\alpha_s)=(-\alpha_s,\dots,-\alpha_1)$, which implies that $\mathcal{C}_T(\beta)=\mathcal{C}_T(\alpha_1+1,\dots,\alpha_s)$ is equal to $\mathcal{C}_T(-\alpha_s+1,\dots,-\alpha_1)$. Moreover, this set is equal to the set $\mathcal{C}_T(\beta')=\mathcal{C}_T(\alpha_1,\dots,\alpha_s-1)$, thus $\mathcal{C}_T(\beta)=\mathcal{C}_T(\beta')$. \\
For the necessary condition, suppose $\mathcal{C}_T(\beta)=\mathcal{C}_T(\beta')$, i.e. $\mathcal{C}_T(\alpha_1+1,\alpha_2,\dots,\alpha_s)=\mathcal{C}_T(\alpha_1,\dots,\alpha_{s-1},\alpha_s-1)$. Thus, since $\alpha_1$ is different from $\alpha_1+1$ and $-\alpha_1-1$, then $(\alpha_1,\dots,\alpha_s-1)$ is equal to one of these tuples:
\begin{enumerate}
\item $(\alpha_i,\alpha_{i+1},\dots,\alpha_s,\alpha_1+1,\alpha_2 \dots,\alpha_{i-1})$ for some $2\leq i \leq s$
\item $(-\alpha_i,-\alpha_{i-1},\dots,-\alpha_2,-\alpha_1-1,-\alpha_s,\dots,-\alpha_{i+1})$ for some $2\leq i \leq s$
\end{enumerate}
Suppose that the first case is true, i.e. $(\alpha_1,\dots,\alpha_s-1)=(\alpha_i,\alpha_{i+1},\dots,\alpha_s,\alpha_1+1,\alpha_2 \dots,\alpha_{i-1})=(\beta_i,\beta_{i+1},\dots,\beta_s,\beta_1,\beta_2 \dots,\beta_{i-1})$ for some $2\leq i \leq s$.\\
First observe that $$\beta_{{\left[i \right]}_s}=\left\lbrace
\begin{array}{ccc}
\alpha_{{\left[i \right]}_s} \ &\mbox{if} \ &1< i \leq s\\
\alpha_1+1 \ &\mbox{if} \ &i=1\\
\end{array} \right..$$\\
We have: $\alpha_1=\alpha_{1_s}=\beta_{\left[ 1+i-1 \right]_s}$ (by Lemma \ref{equalind}) $=\alpha_{\left[ 1+i-1 \right]_s}$ (since otherwise we get $\alpha_1=\alpha_1+1$ which is a contradiction) $=\beta_{\left[ 1+2(i-1) \right]_s}$ (also by Lemma \ref{equalind}) $=\alpha_{\left[ 1+2(i-1) \right]_s}$ (also so that we don't get $\alpha_1=\alpha_1+1$, a contradiction). And so on, we may prove by induction that $$\alpha_1= \alpha_{\left[ 1+k(i-1) \right]_s} \ \forall k \in \mathbb{N^*}, \ \forall i\geq 2. \  (*)$$
Now, observe that $\alpha_1=\alpha_i$, $\alpha_2=\alpha_{i+1}$, $\dots$, $\alpha_{s-i+1}=\alpha_s$ and $\alpha_{s-i+2}=\alpha_1+1$.\\
Moreover, we can write $s-i+2=1+k'(i-1)+\lambda.s=\left[1+k'(i-1)\right]_s$, with $k'=s-1 \in \mathbb{N^*}$ and $\lambda=2-i \in \mathbb{Z}$.\\
It follows that $\alpha_{s-i+2}=\alpha_{\left[ 1+k'(i-1)\right] _s}=\alpha_{1}$ by ($*$). But $\alpha_{s-i+2} = \alpha_1+1$, thus we reach a contradiction. So the first case cannot occur.\\
Now consider the second case. First suppose that $i\neq s$. We have $\alpha_1=-\alpha_i$ for some $2\leq i \leq s-1$, $\alpha_2=-\alpha_{i-1}$, $\alpha_3=-\alpha_{i-2}$, $\dots$, $\alpha_{i-1}=\alpha_{i-((i-1)-1)}=-\alpha_2$ and $\alpha_i=-\alpha_1-1$. Thus $\alpha_1=-\alpha_i=\alpha_1+1$ and we reach a contradiction. So the second case is impossible for $2\leq i \leq s-1$. If $i=s$, we have $\alpha_1=-\alpha_s$, $\alpha_2=-\alpha_{s-1}$, $\alpha_3=-\alpha_{s-2}$ $\dots$,  $\alpha_{s-1}=-\alpha_{s-((s-1)-1)}=-\alpha_2$ and $\alpha_s-1=-\alpha_1-1$ which also means that $\alpha_s=-\alpha_1$. Thus $(\alpha_1,\dots,\alpha_s)=(-\alpha_s,\dots,-\alpha_1)$ and as a result $\alpha$ is symmetric.
\end{proof}

Let $\beta=(\beta_1,\beta_2,\dots,\beta_s)\in\mathbb{Z}^s$; $s$ is even, and $\beta_i\beta_{i+1}<0 \ \forall \ i=1,\dots,s-1$. \\Then $\forall$ $1\leq i \leq s$, define $\beta_i \ast 1$ as:
\begin{equation*}
\beta_i\ast1=
\left\lbrace
\begin{array}{ccc}
\beta_i-1  & \mbox{if} & \beta_1 > 0 \\
\beta_i+1 & \mbox{if} & \beta_1 < 0 \\
\end{array}\right.
\end{equation*}
We are now ready to link between the number of Hamiltonian paths of some type $P(\alpha)$ in $T$ and the number of Hamiltonian cycles of types $C(\beta)$ and $C(\beta')$ (mentioned in the beginning of this section):
\begin{theorem}\label{nonsym}
Let $T$ be a tournament of order $n$, and $(\beta_1,\dots,\beta_s)\in\mathbb{Z}^s;$\\$\sum\limits_{i=1}^{s}\mid\!\beta_i\!\mid = n$, $s$ is even, and $\beta_i\beta_{i+1}<0 \ \forall \ i=1,\dots,s-1$. Then:\\
If $(\beta_1\ast 1, \beta_2,\dots, \beta_s)$ is symmetric, we have:
\begin{equation*}
  f_T(\beta_1\ast 1,\beta_2,\dots,\beta_s)= g_T(\beta_1,\beta_2,\dots,\beta_s).t(\beta_1,\beta_2,\dots,\beta_s).
\end{equation*}
Otherwise, we have:\\
$f_T(\beta_1\ast 1,\beta_2,\dots,\beta_s)= \delta(\beta_1,\beta_2,\dots,\beta_s).g_T(\beta_1,\beta_2,\dots,\beta_s).t(\beta_1,\beta_2,\dots,\beta_s)$\\$+\delta(\beta_1\ast 1,\beta_2,\dots,\beta_s\ast 1).g_T(\beta_1\ast 1,\beta_2,\dots,\beta_s\ast 1).t(\beta_1\ast 1,\beta_2,\dots,\beta_s\ast 1)$\\
\linebreak
where $\delta(\gamma)=
\left\lbrace
\begin{array}{ccc}
1  & \mbox{if} & \gamma \ is \ non-symmetric \ and \ not \ a \ singleton\\
2  & \mbox{if} & \gamma \ \ is \ symmetric\\
\frac{n}{t(\gamma)}  & \mbox{if} & \gamma \ \ is \ a \ singleton \\
\end{array}\right. $
\end{theorem}
\begin{proof}
In order to prove this theorem, let us compute $f_T(\beta_1\ast 1,\dots,\beta_s)$.\\
Consider the set $\mathcal{P}_T(\beta_1\ast 1,\beta_2,\dots,\beta_s)$ and let $P=x_1\dots x_n$ be an element of this set. $C_P$ is either of type $C(\beta_1,\dots,\beta_s)$ or of type $C(\beta_1\ast 1,\beta_2,\dots,\beta_s\ast 1)$ whether $(x_n,x_1)$ or $(x_1,x_n)$ $\in E(T)$. \\
Let $\mathcal{C}_T(\beta)=\lbrace C_1,\dots,C_t\rbrace$ be the set of cycles of type $C(\beta_1,\dots,\beta_s)$ in $T$, and let $\mathcal{C}_T(\beta')=\lbrace C'_1,\dots,C'_r\rbrace$ be the set of cycles of type $C(\beta_1\ast 1,\beta_2,\dots,\beta_s\ast 1)$ in $T$.
We have two cases to consider:
\begin{enumerate}
\item[(a)] If $(\beta_1\ast 1,\beta_2,\dots,\beta_s)$ is symmetric, then  by Theorem \ref{EqSym}, $\mathcal{C}_T(\beta)= \mathcal{C}_T(\beta')$. Thus we only have to consider one of them, say $\mathcal{C}_T(\beta)$, to avoid counting the same cycle twice in the following step.\\
Let $\mathcal{C}_T(\beta)=\lbrace C_1,C_2,\dots,C_t \rbrace$. We have that for all $C_i \in \mathcal{C}_T(\beta)$, there exists a subclass $X_i$ of $\mathcal{P}_T(\beta_1\ast 1,\beta_2,\dots,\beta_s)$ with respect to $\mathcal{R}_{(\beta_1\ast 1,\beta_2,\dots,\beta_s)}$ such that every path in $X_i$ generates $C_i$. Thus by Theorem \ref{EquClass}, $\mid X_i \mid = \mid \overline{P} \mid = t(\beta)$ for some $P\in X_i$, since if $\beta_1\ast 1,\beta_2,\dots,\beta_s$ is symmetric, none of $\beta$ or $\beta'$ can be symmetric, nor a singleton. Hence, 
\begin{eqnarray*}
f_T(\beta_1\ast 1,\beta_2,\dots,\beta_s)= &&\sum_{i=1}^t \mid X_i \mid  
\ =  \ \sum_{i=1}^t t(\beta)\\
 = && t.t(\beta) 
\  = \  \mid \mathcal{C}_T(\beta) \mid . t(\beta) \\
 =&&  g_T(\beta).t(\beta).
 \end{eqnarray*}

\item[(b)] If $\beta_1\ast 1,\beta_2,\dots,\beta_s$ is non-symmetric, then by Theorem \ref{EqSym}, $\mathcal{C}_T(\beta)\neq \mathcal{C}_T(\beta')$, thus $\mathcal{C}_T(\beta)\cap \mathcal{C}_T(\beta')=\emptyset$ (because the sets of every type of Hamiltonian cycles form a partition of the set of all oriented Hamiltonian cycles in $T$).\\
For all $C_i \in \mathcal{C}_T(\beta)$, there exists a subclass $X_i$ of $\mathcal{P}_T(\beta_1\ast 1,\beta_2,\dots,\beta_s)$ with respect to $\mathcal{R}_{(\beta_1\ast 1,\beta_2,\dots,\beta_s)}$ such that every path in $X_i$ generates $C_i$. Thus by Theorem \ref{EquClass}, and Remark \ref{EquClass1.5}, $\mid X_i \mid = \mid \overline{P} \mid = t(\beta)$ or $2.t(\beta)$ or $n$ for some $P\in X_i$, whether $\beta$ is non-symmetric and not a singleton, is symmetric, or is a singleton, so $\mid X_i \mid = \mid \overline{P} \mid=\delta(\beta).t(\beta)$. \\ Similarly, $\forall$ $C'_j \in \mathcal{C}_T(\beta')$, $\exists$ a subclass $X'_j$ of $\mathcal{P}_T(\beta_1\ast 1,\beta_2,\dots,\beta_s)$ with respect to $\mathcal{R}_{(\beta_1\ast 1,\beta_2,\dots,\beta_s)}$ such that every path in $X'_j$ generates $C'_j$. Thus $\mid X'_j \mid = \mid \overline{P'} \mid =\delta(\beta'). t(\beta')$ for some $P'\in X'_j$.\\
Hence, 
\begin{eqnarray*}
f_T(\beta_1\ast 1,\beta_2,\dots,\beta_s)= &&\sum_{i=1}^t \mid X_i \mid + \sum_{j=1}^r \mid X'_j \mid \\
= && \sum_{i=1}^t \delta(\beta).t(\beta)+ \sum_{j=1}^r \delta(\beta').t(\beta') \\
= && t.\delta(\beta).t(\beta)+r.\delta(\beta').t(\beta') \\
= && \mid \mathcal{C}_T(\beta) \mid . \delta(\beta).t(\beta)+\mid \mathcal{C}_T(\beta') \mid . \delta(\beta').t(\beta') \\ 
= && g_T(\beta).\delta(\beta).t(\beta)+g_T(\beta').\delta(\beta').t(\beta'),
\end{eqnarray*}
and this concludes our proof.
\end{enumerate}
\end{proof} 
Moreover, for the case when the type of oriented paths in a tournament $T$ is symmetric, we have the following property:

\begin{theorem}\label{t=1}
Let $\alpha=(\alpha_1,\dots,\alpha_s)\in \mathbb{Z}^s$, $\alpha_i \cdot \alpha_{i+1}< 0$, $\alpha$ symmetric, and $T$ a tournament of order $n=\sum\limits_{i=1}^{s}\mid\alpha_i\mid +1$. Let $P$ be a Hamiltonian path in $T$ of type $P(\alpha)$ and $C_P\in \mathcal{C}_T(\beta)$. Then we have: $$t(\beta)=1.$$
\end{theorem}

\begin{proof}
Suppose that $\alpha_1>0$. Since $\alpha$ is symmetric, then $s$ is even, and by Theorem \ref{EqSym}, we can assume that $C_P\in \mathcal{C}_T(\alpha_1+1,\dots,\alpha_s)=\mathcal{C}_T(\beta)$. If $\alpha_1<0$, then also by Theorem \ref{EqSym}, we can assume that $C_P\in \mathcal{C}_T(\alpha_1-1,\dots,\alpha_s)$, but we will treat the case $\alpha_1>0$, and the other case is similar.\\
Since $\alpha$ is symmetric then $\alpha=(\alpha_1,\alpha_2,\dots,\alpha_l,-\alpha_l,\dots,-\alpha_2,-\alpha_1)$ where $l=\frac{s}{2}$, and $\beta=(\alpha_1+1,\alpha_2,\dots,\alpha_l,-\alpha_l,\dots,-\alpha_2,-\alpha_1)$.\\
Set $r'=r(\beta)$, we have $t(\beta)=\frac{s}{r'}$. Suppose to the contrary that $t(\beta)>1$, We have 2 cases:\\
\begin{enumerate}
\item[(a)] $t(\beta)$ is even. Set $t(\beta)=2k$, thus $\beta$ is divided into $2k$ tuples $(\beta_1,\dots,\beta_{r'})$.\\
Set $a$ be the first component of the first tuple, we have $a=\alpha_1+1$. Set $b$ be the last component of the last tuple ($2k^{th}$ tuple), we have $b=-\alpha_1$.\\
Since $r'$ is a period, then the first component $a'$ of the $(k+1)^{th}$ tuple is equal to $a$, and the last component $b'$ of the $k^{th}$ tuple is equal to $b$.\\
But since $\alpha$ is symmetric, $a'=-b'$ because $a'=\alpha_l$ and $b'=-\alpha_l$. Thus $a=-b$ which implies that $\alpha_1+1=-(-\alpha_1)=\alpha_1$ and this leads to a contradiction. So $t(\beta)$ cannot be even.
\item[(b)] $t(\beta)$ is odd. Set $t(\beta)=2k+1$, $k\geq 1$, thus $\beta$ is divided into $2k+1$ tuples $(\beta_1,\dots,\beta_{r'})$, by noting that since $\alpha$ is symmetric, the $(k+1)^{th}$ tuple should be of the form $(\beta_1,\dots,\beta_{\frac{r'}{2}},-\beta_{\frac{r'}{2}},\dots,-\beta_1)$ where $\beta_{\frac{r'}{2}}=\alpha_l$, thus it is symmetric. (Obviously all the other $2k$ tuples have this form since they are all equal). \\
Set $a$ be the first component of the first tuple, we have $a=\alpha_1+1$. Set $b$ be the last component of the last tuple ($(2k+1)^{th}$ tuple), we have $b=-\alpha_1$.\\
Since $r'$ is a period, then the first component $a'$ of the $(k+1)^{th}$ tuple is equal to $a$, and the last component $b'$ of the $(k+1)^{th}$ tuple is equal to $b$.\\
But since the $(k+1)^{th}$ tuple is symmetric, $a'=-b'$. Thus $a=-b$ which implies that $\alpha_1+1=-(-\alpha_1)=\alpha_1$ and this leads to a contradiction. So $T(\beta)$ cannot be an odd integer strictly greater than  1.\\  Thus we conclude that $t(\beta)=1$.
\end{enumerate}
\end{proof}
And finally, with the same hypothesis of Theorem \ref{nonsym}, we can deduce the following:
\begin{corollary}\label{not}
If $(\beta_1\ast 1, \beta_2,\dots, \beta_s)$ is symmetric, Then:
\begin{equation*}
  f_T(\beta_1\ast 1,\beta_2,\dots,\beta_s)= g_T(\beta_1,\beta_2,\dots,\beta_s)
\end{equation*}
\end{corollary}

\begin{proof}
The result follows immediately from Theorem \ref{nonsym} and Theorem \ref{t=1}.
\end{proof}
\bigskip

\section{Oriented Hamiltonian cycles}

Based on Theorem \ref{nonsym}, linking between the number of oriented Hamiltonian paths of some type, and the number of oriented Hamiltonian cycles that can be generated by these paths in a tournament, we are now able to establish the main result of Section 2, Theorem \ref{f(a)=f(-a)}, for oriented cycles:
\begin{theorem}\label{g(a)=g(-a)}
Let $\alpha=(\alpha_1,\dots,\alpha_s)\in \mathbb{Z}^s$; $\alpha_i \cdot \alpha_{i+1} <0$, $\alpha_1\geq 0$, $s$ is even if $s\neq 1$, and let $T$ be a tournament of order $n$; $n=\sum\limits_{i=1}^s \mid \alpha_i \mid$. We have:\\
$$g_T(\alpha)=g_T(-\alpha).$$
\end{theorem}

\begin{proof}
The proof will be done by induction on $s$.\\
\linebreak
If $s=1$, $\alpha=(\alpha_1)=(n)$ and $-\alpha=(-\alpha_1)=(-n)$ and we have $g_T(n)=g_T(-n)$.\\
Suppose that the result is true for $s-2$ blocks, $s>2$. That is, if $\alpha=(\alpha_1,\dots,\alpha_{s-2})\in \mathbb{Z}^{s-2}$; $\alpha_i \cdot \alpha_{i+1} <0$, $\alpha_1\geq 0$, $s-2$ is even, and $T$ is a tournament of order $n$; $n=\sum\limits_{i=1}^{s-2} \mid \alpha_i \mid$, we have: $g_T(\alpha)=g_T(-\alpha)$. Let's prove the result for $s$ blocks. Let $\alpha=(\alpha_1,\dots,\alpha_{s})\in \mathbb{Z}^{s}$; $\alpha_i \cdot \alpha_{i+1} <0$, $\alpha_1\geq 0$, and let $T$ be a tournament of order $n$; $n=\sum\limits_{i=1}^{s} \mid \alpha_i \mid$. We argue by induction on $\alpha_1$. \\
\linebreak
If $\alpha_1=0$, then by induction, $g_T(\alpha)=g_T(0,\alpha_2,\dots,\alpha_{s})=g_T(\alpha_2+\alpha_{s},\alpha_3,\dots,\alpha_{s-1})=g_T(-\alpha_2-\alpha_{s},-\alpha_3,\dots,-\alpha_{s-1})=g_T(0,-\alpha_2,-\alpha_3,\dots,-\alpha_{s-1},-\alpha_{s})=g_T(-\alpha)$.\\
So suppose that $\alpha_1>0$ and the result is true when the first block is of length $\alpha_1-1$, and let's prove it when the first block is of length $\alpha_1$. \\
\linebreak
We will consider two cases:
\begin{enumerate}
\item[(a)] The tuple $(\alpha_1-1,\alpha_2,\dots,\alpha_{s})$ is non-symmetric.\\
We have $\alpha_1-1\geq 0$, thus by Theorem \ref{nonsym}, 
\begin{eqnarray*}
f_T(\alpha_1-1,\alpha_2,\dots,\alpha_s)= && \delta(\alpha_1,\dots,\alpha_s).t(\alpha_1,\dots,\alpha_s).g_T(\alpha_1,\dots,\alpha_s) \\
&& +\delta(\alpha_1 -1,\alpha_2,\dots,\alpha_{s-1},\alpha_s-1).t(\alpha_1 -1,\alpha_2,\dots,\alpha_{s-1},\alpha_s-1) \\
&& .g_T(\alpha_1 -1,\alpha_2,\dots,\alpha_{s-1},\alpha_s-1) \\
= && \delta(\beta_1).t(\beta_1).g_T(\beta_1)+\delta(\beta'_1).t(\beta'_1).g_T(\beta'_1)
\end{eqnarray*}
where $\delta(\gamma)=
\left\lbrace
\begin{array}{ccc}
1  & \mbox{if} & \gamma \ is \ non-symmetric \ and  \ not \ a \ singleton\\
2  & \mbox{if} & \gamma \ \ is \ symmetric\\
\frac{n}{t(\gamma)}  & \mbox{if} & \gamma \ \ is \ a \ singleton \\
\end{array}\right. $ \\
\linebreak
\noindent Now consider the tuple $(-\alpha_1+1,-\alpha_2,\dots,-\alpha_{s})$ which is also non-symmetric. We have $-\alpha_1+1\leq 0$, thus by Theorem \ref{nonsym}, 
\begin{eqnarray*}
f_T(-\alpha_1+1,-\alpha_2,\dots,-\alpha_s)= && \delta(-\alpha_1,\dots,-\alpha_s).t(-\alpha_1,\dots,-\alpha_s).g_T(-\alpha_1,\dots,-\alpha_s) \\
&& +\delta(-\alpha_1+1 ,-\alpha_2,\dots,-\alpha_s+1).t(-\alpha_1+1 ,-\alpha_2,\dots,-\alpha_s+1) \\
&& .g_T(-\alpha_1+1 ,-\alpha_2,\dots,-\alpha_s+1) \\
=&& \delta(\beta_2).t(\beta_2).g_T(\beta_2)+\delta(\beta'_2).t(\beta'_2).g_T(\beta'_2)
\end{eqnarray*}
where $\delta(\gamma)=
\left\lbrace
\begin{array}{ccc}
1  & \mbox{if} & \gamma \ is \ non-symmetric \ and \ not \ a \ singleton\\
2  & \mbox{if} & \gamma \ \ is \ symmetric\\
\frac{n}{t(\gamma)}  & \mbox{if} & \gamma \ \ is \ a \ singleton \\
\end{array}\right. $
\\
\linebreak
Since $(-\alpha_1+1,-\alpha_2,\dots,-\alpha_s)=-(\alpha_1-1,\alpha_2,\dots,\alpha_s)$, then by Theorem \ref{f(a)=f(-a)} we have $f_T(\alpha_1-1,\alpha_2,\dots,\alpha_s)=f_T(-\alpha_1+1,-\alpha_2,\dots,-\alpha_s)$. As a result, $$\delta(\beta_1).g_T(\beta_1).t(\beta_1)+\delta(\beta'_1).g_T(\beta'_1).t(\beta'_1)=\delta(\beta_2).g_T(\beta_2).t(\beta_2)+\delta(\beta'_2).g_T(\beta'_2).t(\beta'_2).$$
But, since $\beta_2=-\beta_1$ and $\beta'_2=-\beta'_1$ thus if $\beta_1$ is non-symmetric and not a singleton (resp. is a singleton, or is symmetric), so is $\beta_2$, and similarly for $\beta'_1$ and $\beta'_2$, so $\delta(\beta_1)=\delta(\beta_2)$, and $\delta(\beta'_1)=\delta(\beta'_2)$, and also by Proposition \ref{t} we have $t(\beta_2)=t(\beta_1)$ and $t(\beta'_2)=t(\beta'_1)$. \\
Moreover, since $\alpha_1-1 < \alpha_1$, then by induction $g_T(\beta'_1)=g_T(\beta'_2)$, hence we have $$g_T(\beta_1)=g_T(\beta_2),$$ and the result follows.\\
\item[(b)] The tuple $(\alpha_1-1,\alpha_2,\dots,\alpha_{s})$ is symmetric.\\
We have $\alpha_1-1\geq 0$, thus by Corollary \ref{not}, $$f_T(\alpha_1-1,\alpha_2,\dots,\alpha_s)=g_T(\alpha_1,\alpha_2,\dots,\alpha_s).$$
Now consider the tuple $(-\alpha_1+1,-\alpha_2,\dots,-\alpha_{s})$ which is also symmetric. We have $-\alpha_1+1\leq 0$, thus by Corollary \ref{not}, $$f_T(-\alpha_1+1,-\alpha_2,\dots,-\alpha_s)=g_T(-\alpha_1,-\alpha_2,\dots,-\alpha_s).$$
Since by Theorem \ref{f(a)=f(-a)} we have $$f_T(\alpha_1-1,\alpha_2,\dots,\alpha_s)=f_T(-\alpha_1+1,-\alpha_2,\dots,-\alpha_s),$$
we get $$g_T(\alpha_1,\dots,\alpha_s)=g_T(-\alpha_1,\dots,-\alpha_s).$$
\end{enumerate}
\end{proof}

\section{Digraphs of maximal degree $\Delta \leq 2$}

After establishing Theorem \ref{f(a)=f(-a)} and Theorem \ref{g(a)=g(-a)}, proving that a tournament and its complement contain the same number of oriented Hamiltonian paths and cycles of any given type, we may generalize this fact to any digraph of maximal degree 2: If $H$ is a digraph with maximal degree $\Delta(G(H))\leq 2$, then $f_T(H)=f_{\overline{T}}(H)$, where $f_T(H)$ is the number of copies of the digraph $H$ in a tournament $T$.\\
\linebreak
For this purpose, we first need to prove several lemmas:

\begin{lemma}\label{1}
Let $\alpha=(\alpha_1,\dots,\alpha_s)\in \mathbb{Z}^s$; $\alpha_i \cdot \alpha_{i+1} <0$, $\alpha_1 \geq 0$, and let $T$ be a tournament of order $n$; $n\geq \sum\limits_{i=1}^s \mid \alpha_i \mid +1$. We have:\\
$$f_T(\alpha)=f_T(-\alpha).$$
\end{lemma}

\begin{proof}
Let $m=\sum\limits_{i=1}^s \mid \alpha_i \mid +1$. Every oriented path in $T$ of type $P(\alpha)$ is a Hamiltonian path of type $P(\alpha)$ contained in a subtournament $T'$ of $T$ of order $m$. By Theorem \ref{f(a)=f(-a)}, $f_{T'}(\alpha)=f_{T'}(-\alpha)$. Moreover, if we consider another subtournament $T''$ of $T$, of order $m$, $T'' \neq T'$, then $\mathcal{P}_{T'}(\alpha) \cap \mathcal{P}_{T''}(\alpha)= \emptyset$, because every Hamiltonian path in $T'$ differs with a least one vertex from every Hamiltonian path in $T''$. \\So let $V(T)=\bigcup\limits_{X\subseteq V(T), \ |X|=m} X$, we have: $$f_T(\alpha)=\sum\limits\limits_{X\subseteq V(T), \ |X|=m} f_{\langle X\rangle}(\alpha)=\sum\limits_{X\subseteq V(T), \ |X|=m} f_{\langle X\rangle}(-\alpha)=f_T(-\alpha),$$
and we get our result.
\end{proof}

Similarly, we may prove the same result for cycles in tournaments:

\begin{lemma}\label{2}
Let $\alpha=(\alpha_1,\dots,\alpha_s)\in \mathbb{Z}^s$; $\alpha_i \cdot \alpha_{i+1} <0$, $\alpha_1 \geq 0$, and let $T$ be a tournament of order $n$; $n\geq \sum\limits_{i=1}^s \mid \alpha_i \mid$. We have:
$$g_T(\alpha)=g_T(-\alpha).$$
\end{lemma}

\begin{lemma}\label{Imp}
Let $T$ be a tournament, and let $H$ be a digraph with $\Delta(G(H))\leq 2$ and such that its connected components are mutually isomorphic. Then the number of copies of $H$ in $T$ and that in its complement $\overline{T}$ are the same.
\end{lemma}

\begin{proof}
Since $H$ is a digraph with $\Delta(G(H))\leq 2$ and such that its connected components are isomorphic, then $H=H_1\cup H_2\cup \dots \cup H_r$ where the digraphs $H_i$, $1\leq i \leq r$, are its connected components, with $|H_i|=m$ $\forall$ $1\leq i \leq r$, and such that they are either all paths of the same type, say $P(\alpha)$, or all cycles of the same type, $C(\beta)$. If $T$ contains a copy of $H$, then since the digraphs $H_i$, $1\leq i \leq r$, are disjoint, the copy of every digraph $H_i$ is a spanning subdigraph of a subtournament $T_i$ of $T$, such that the subtournaments $T_i$, $1\leq i \leq r$, are also disjoint, with $|V(T_i)|=m$ $\forall$ $1\leq i \leq r$. Note that $rm \leq n$.\\
Let's consider $r$ disjoint subtournaments of $T$, $T_i$ , $1\leq i \leq r$, all of order $m$, and suppose that $T$ contains a copy of $H$ such that $\forall$ $1\leq i \leq r$, $H_i$ has a copy in $T_i$. As $f_{T_i}(H_i)$ denotes the number of copies of $H_i$ in the subtournament $T_i$, then the number of copies of $H$ in $T$, such that the copy of $H_i$ is a spanning subdigraph of $T_i$, is: $$\prod_{i=1}^r f_{T_i}(H_i).$$
Now if we consider any permutation $\sigma$ of the subtournaments $T_i$, and since all the digraphs $H_i$ are isomorphic, then $\forall$ $1\leq i \leq r$, if $H_i$ has a copy in $T_i$, then $H_i$ also has a copy in $T_{\sigma(i)}$.
But, also since all the digraphs $H_i$ are isomorphic, then the copies of $H$ obtained in $T$ such that the copy of each $H_i$ is a spanning subdigraph of $T_i$ are the same as the ones obtained in $T$  such that the copy of each $H_i$ is a spanning subdigraph of $T_{\sigma(i)}$.\\
Let's compute now $f_T(H)$, the total number of copies of $H$ in $T$. \\
Let $\mathcal{L}=\lbrace ( T_1,T_2,\dots,T_r );$ $T_i$ subtournament of $T$ $\forall$ $1\leq i \leq r$, $T_i\cap T_j=\emptyset$ $\forall$ $1\leq i,j \leq r,$ $|V(T_i)|=m\rbrace$. We have:
$$f_T(H)=\sum\limits_{( T_1,T_2,\dots,T_r )\in \mathcal{L}} \frac{\prod_{i=1}^r f_{T_i}(H_i)}{r!}.$$
However, by Lemma \ref{1} and Lemma \ref{2}, we have that $\forall$ $1\leq i \leq r$, $$f_{T_i}(H_i)=f_{\overline{T_i}}(H_i).$$
So let $\mathcal{L'}=\lbrace ( \overline{T_1},\overline{T_2},\dots,\overline{T_r} );$ $T_i$ subtournament of $T$ $\forall$ $1\leq i \leq r$, $T_i\cap T_j=\emptyset$ $\forall$ $1\leq i,j \leq r,$ $|V(T_i)|=m\rbrace$, we get: $$f_T(H)=\sum\limits_{( T_1,T_2,\dots,T_r )\in \mathcal{L}} \frac{\prod_{i=1}^r f_{T_i}(H_i)}{r!}=\sum\limits_{( \overline{T_1},\overline{T_2},\dots,\overline{T_r} )\in \mathcal{L'}} \frac{\prod_{i=1}^r f_{\overline{T_i}}(H_i)}{r!}=f_{\overline{T}}(H),$$
and the result follows.
\end{proof}

\noindent We may now prove the main result of this section:

\begin{theorem}
Let $T$ be a tournament and let $H$ be a digraph with $\Delta(G(H))\leq 2$. Then the number of copies of $H$ in $T$ and its complement $\overline{T}$ is the same.
\end{theorem}

\begin{proof}
Since $\Delta(G(H))\leq 2$, then $H$ is a disjoint union of directed paths and cycles.
Write $H$ as $H= \bigcup_{i=1}^t H^i$, where each $H^i$ is a subdigraph of $H$ whose all connected components are isomorphic, and which is maximal with this property. The connected components of each $H^i$ are either all paths of the same type or cycles of the same type. Note that the digraphs $H^i$, $1\leq i \leq t$, are disjoint, and non-isomorphic.\\
If $T$ contains a copy of $H$, then since the digraphs $H^i$, $1\leq i \leq t$, are disjoint, the copy of every digraph $H^i$ is a spanning subdigraph of a subtournament $T^i$ of $T$, and such that the subtournaments $T^i$, $1\leq i \leq t$, are also disjoint, with $|V(T^i)|=|V(H^i)|$ $\forall$ $1\leq i \leq t$. \\
As we did in the previous lemma, let's consider $t$ disjoint subtournaments of $T$, $T^i$, $1\leq i \leq t$, and such that $|V(T^i)|=|V(H^i)|$, and suppose that $T$ contains a copy of $H$ such that $\forall$ $1\leq i \leq t$, $H^i$ has a copy in $T^i$. The number of copies of $H$ in $T$, such that the copy of $H^i$ is a spanning subdigraph of $T^i$, is: $$\prod_{i=1}^t f_{T^i}(H^i).$$
However, if we consider any permutation $\sigma$ of the subtournaments $T^i$, and since all the digraphs $H^i$ are non-isomorphic, then if $T$ contains a copy of $H$ such that $\forall$ $1\leq i \leq t$, $H^i$ has a copy in $T^{\sigma(i)}$, the copies of $H$ obtained in $T$ such that the copy of each $H^i$ is a spanning subdigraph of $T^i$ are all different from those obtained in $T$  such that the copy of each $H^i$ is a spanning subdigraph of $T^{\sigma(i)}$.\\
\linebreak 
So let's compute now the total number of copies of $H$ in $T$,  $f_T(H)$. \\
\linebreak
Let $\mathcal{L}=\lbrace ( T^1,T^2,\dots,T^t );$ $T^i$ subtournament of $T$ $\forall$ $1\leq i \leq t$, $T^i\cap T^j=\emptyset$ $\forall$ $1\leq i,j \leq t,$ $|V(T_i)|=|V(H^i)|\rbrace$. We have:
$$f_T(H)=\sum\limits_{( T^1,T^2,\dots,T^t )\in \mathcal{L}} \prod_{i=1}^t f_{T^i}(H^i).$$
\linebreak
However, by Lemma \ref{Imp}, since the connected components of each digraph $H^i$ are isomorphic, we have that $\forall$ $1\leq i \leq t$, $f_{T^i}(H^i)=f_{\overline{T^i}}(H^i)$.\\
So let $\mathcal{L'}=\lbrace ( \overline{T^1},\overline{T^2},\dots,\overline{T^t} );$ $T^i$ subtournament of $T$ $\forall$ $1\leq i \leq t$, $T^i\cap T^j=\emptyset$ $\forall$ $1\leq i,j \leq t,$ $|V(T^i)|=|V(H^i)|\rbrace$, we get:
$$f_T(H)=\sum\limits_{( T^1,T^2,\dots,T^t )\in \mathcal{L}} \prod_{i=1}^t f_{T^i}(H^i)=\sum\limits_{( \overline{T^1},\overline{T^2},\dots,\overline{T^t} )\in \mathcal{L'}} \prod_{i=1}^t f_{\overline{T^i}}(H^i)=f_{\overline{T}}(H),$$
hence:
$$f_T(H)=f_{\overline{T}}(H),$$
and this concludes the proof.
\end{proof}
\begin{remark}
Let $T$ be a tournament on $n+1$ vertices, formed by a directed $n$-cycle $C=v_1v_2\dots v_{n}$, with its internal edges, where these edges may have any orientations, and a vertex $v$ of in-degree equal to zero (a source), adjacent to the $n$ vertices of the cycle ($d^+_T(v)=n$). Then the complement $\overline{T}$ of this tournament is formed by a directed $n$-cycle, $C'=v_1v_nv_{n-1}\dots v_2$, and its internal edges which have opposite orientations of those of $T\langle C\rangle$, and a vertex $v$ of out-degree equal to zero (a sink) adjacent to all the vertices of $C'$. Also note that since $C$ and $C'$ are directed cycles, then $\forall$ $x\in C$, $d^+_T(x)\leq n-1$ and $\forall$ $y\in C'$, $d^+_{\overline{T}}(y)\leq n-1$.\\
Thus if we consider a digraph $H$ on $n+1$ vertices, formed by a vertex $y$ and $n$ out-neighbors of $y$, which is a digraph of maximal degree $\Delta(G(H))=n$, the number of copies of $H$ in $T$ is equal to one, while there are no such copies in $\overline{T}$.
\end{remark}

Based on the remark above, we finally ask the following:\\
Let $f_T(H)$ denote the number of copies of a digraph $H$ in a tournament $T$.

\begin{problem}
Can we characterize the set $\mathcal{H}$ of all digraphs $H$ such that $f_T(H)=f_{\overline{T}}(H)$ for any tournament $T$?
\end{problem}

\bigskip
\bigskip

\noindent \textbf{Acknowledgments.}
We would like to thank the Lebanese University for the PhD grant, and Campus France for the Eiffel excellence scholarship (Eiffel 2018).

 \medskip
 
\end{document}